\documentclass[11pt]{amsart}

\usepackage{amsmath}
\usepackage{amsfonts}
\usepackage{amssymb}
\usepackage{amscd}
\usepackage{amsthm}
\usepackage{framed}
\usepackage{fullpage}
\usepackage{graphicx}
\usepackage{latexsym}
\usepackage[numbers]{natbib}

\usepackage{hyperref}
\hypersetup{colorlinks,linkcolor={red},citecolor={blue},urlcolor={blue}}

\newtheorem{theorem}{Theorem}[section]
\newtheorem{lemma}[theorem]{Lemma}
\newtheorem{proposition}[theorem]{Proposition}
\newtheorem{corollary}[theorem]{Corollary}

\theoremstyle{definition}
\newtheorem{definition}[theorem]{Definition}

\newtheorem{remark}[theorem]{Remark}

\newtheoremstyle{TheoremRef}
        {8.0pt plus 2.0pt minus 4.0pt}
        {8.0pt plus 2.0pt minus 4.0pt}  
        {\itshape}                      
        {}                              
        {\bfseries}                     
        {.}                             
        {5pt}                             
        {\thmname{#1}\thmnote{ \bfseries #3}}
\theoremstyle{TheoremRef}

\numberwithin{equation}{section}

\newcommand{\CC}{\mathbb C}
\newcommand{\HH}{\mathbb H}

\newcommand{\cD}{\mathcal D}

\newcommand{\QQ}{\mathbb Q}
\newcommand{\RR}{\mathbb R}
\newcommand{\ZZ}{\mathbb Z}

\newcommand{\Orth}{\mathop{\null\mathrm {O}}\nolimits}

\newcommand{\latt}[1]{{\langle{#1}\rangle}}

\newenvironment{psmallmatrix}
  {\left(\begin{smallmatrix}}
{\end{smallmatrix}\right)}

\begin{document}

\title{Automorphic products that are singular modulo primes}

\author{Haowu Wang}

\address{School of Mathematics and Statistics, Wuhan University, Wuhan 430072, Hubei, People's Republic of China}

\email{haowu.wangmath@whu.edu.cn}

\author{Brandon Williams}

\address{Lehrstuhl A für Mathematik, RWTH Aachen, 52056 Aachen, Germany}

\email{brandon.williams@matha.rwth-aachen.de}

\date{\today}

\begin{abstract}
We use Rankin--Cohen brackets on $\mathrm{O}(n, 2)$ to prove that the Fourier coefficients of reflective Borcherds products often satisfy congruences modulo certain primes.
\end{abstract}

\maketitle

\section{Introduction}
This note is inspired by the paper \cite{KKN}, in which it was observed that most of the Fourier coefficients of the (suitably normalized) Siegel cusp form $\Phi_{35}$ of degree two and weight $35$ are divisible by the prime $p=23$. More precisely, if one writes 
$$
\Phi_{35}(Z) = \sum_{T > 0} a(T) e^{2\pi i \mathrm{Tr}(TZ)}, \quad Z \in \mathbb{H}_2,
$$ 
the sum extending over positive-definite half-integral $(2 \times 2)$-matrices $T$, then the main result of \cite{KKN} is that \begin{equation}\label{eq:cong}
a(T) \not \equiv 0 \, (\text{mod}\, 23) \; \Rightarrow \; \mathrm{det}(T) \equiv 0 \, (\text{mod}\, 23).
\end{equation}

This has already been generalized in several ways. In \cite{T}, similar congruences are derived for Siegel cusp forms of higher weights. The papers \cite{KN, NT} prove analogous results for Hermitian modular forms of degree two over the Gaussian and Eisenstein integers. The paper \cite{N} considers \emph{quaternionic} modular forms of degree two, while \cite{BKN, N2, NT2} consider Siegel modular forms of general degree.  We call modular forms satisfying congruences of type \eqref{eq:cong} \textit{singular modulo $p$}. 

In this note, we start with the fact that the cusp form $\Phi_{35}$ is a \emph{reflective Borcherds product} \cite{Bor98, Bor95, GN96, GN98}, which in this situation means that it vanishes only on Humbert surfaces in the  Siegel upper half-space that are fixed by transformations in the Siegel modular group. A natural generalization is to consider reflective Borcherds products on general orthogonal groups $\mathrm{O}(n, 2)$, with Siegel modular forms appearing through the exceptional isogeny from $\mathrm{Sp}_4(\mathbb{R})$ to $\mathrm{O}(3, 2)$. 

It turns out that reflective Borcherds products on $\Orth(n,2)$ \emph{with simple zeros} and of weight $k$ are very often singular modulo primes $p$ dividing $n/2-1-k$. In this note, we give a general argument to prove singularity modulo $p$ that takes a set of \emph{two or more} reflective Borcherds products and proves that some of them are singular modulo specific primes, using an identity based on the Rankin--Cohen brackets on $\mathrm{O}(n, 2)$. This argument requires almost no computation: the presence of congruences such as (\ref{eq:cong}) for $\Phi_{35}$ can be deduced from the location of its zeros. We also give similar arguments that can be used to prove that a single reflective product is singular modulo certain primes.

This note is organized as follows. In \S \ref{sec:background} we review reflective modular forms and define what it means for a modular form to be singular modulo a prime $p$.  In \S \ref{sec:construction} we introduce the Rankin--Cohen bracket on $\Orth(n,2)$ and explain how to use it to derive modular forms that are singular modulo primes. In the last two sections we work out over $50$ reflective Borcherds products that are singular modulo primes. In particular, for every prime $p<60$, we construct at least one mod $p$ singular modular form.  

\section{Reflective modular forms and singular modular forms modulo primes}\label{sec:background}

Let $L$ be an even integral lattice of signature $(n, 2)$ with $n \ge 3$, and let $L_{\mathbb{R}} = L \otimes \mathbb{R}$ and $L_{\mathbb{C}} = L \otimes \mathbb{C}$. The $\mathbb{Z}$-valued quadratic form on $L$ is denoted by $Q$ and the even bilinear form is $$\langle x, y \rangle = Q(x+y) - Q(x) - Q(y), \quad x, y \in L.$$

Attached to the orthogonal group $\mathrm{O}(L_{\mathbb{C}})$ is the Hermitian symmetric domain $\mathcal{D}$, the Grassmannian of oriented negative-definite planes in $L_{\mathbb{R}}$. This is naturally identified with one of the two connected components of $$\{[\mathcal{Z}] \in \mathbb{P}^1(L_{\mathbb{C}}): \; \langle \mathcal{Z}, \mathcal{Z} \rangle = 0, \; \langle \mathcal{Z}, \overline{\mathcal{Z}} \rangle < 0\}$$ by identifying $[\mathcal{X}+i\mathcal{Y}] \in \mathbb{P}^1(L_{\mathbb{C}})$ with the plane through $\mathcal{X}$ and $\mathcal{Y}$. We denote by $\Orth^+(L)$ the orthogonal subgroup that fixes both $\cD$ and $L$. 

Let $\Gamma \le \mathrm{O}^+(L)$ be a finite-index subgroup and $\chi : \Gamma \rightarrow \mathbb{C}^{\times}$ a character. A modular form of integral weight $k$, level $\Gamma$ and character $\chi$ is a holomorphic function $F$ on the cone over $\mathcal{D}$, $$\mathcal{A} = \{\mathcal{Z} \in L_{\mathbb{C}}: \; [\mathcal{Z}] \in \mathcal{D}\}$$ that satisfies the functional equations $$F(\gamma \mathcal{Z}) = \chi(\gamma) F(\mathcal{Z}) \quad \text{and} \quad F(t\mathcal{Z}) = t^{-k} F(\mathcal{Z})$$ for every $\gamma \in \Gamma$ and $t \in \mathbb{C}^{\times}$. 

\begin{definition}\label{def:reflective}
A non-constant modular form $F$ for $\Gamma\leq \Orth^+(L)$ is called \textit{reflective} if its zeros lie on hyperplanes
$$
r^\perp = \{ [\mathcal{Z}] \in \mathcal{D} : \latt{\mathcal{Z}, r} = 0 \}
$$
whose associated reflections 
$$
\sigma_r : L_\RR \to L_\RR, \quad  v \mapsto v - \frac{2\latt{v,r}}{\latt{r,r}}r
$$
lie in $\Gamma$.    
\end{definition} 

Reflective modular forms were introduced by Borcherds \cite{Bor98} and Gritsenko--Nikulin \cite{GN98} in 1998 and they have applications to generalized Kac--Moody algebras, hyperbolic reflection groups and birational geometry. The above definition is somewhat stronger than that of \cite{GN98}, where $F$ is called reflective if the reflections corresponding to zeros of $F$ lie in the larger group $\Orth^+(L)$. Bruinier's converse theorem \cite{Bru02} shows that, in many cases, all reflective modular forms can be constructed through the multiplicative Borcherds lift \cite{Bor95, Bor98}. In this case, the Fourier series of a reflective form has a natural infinite product expansion in which the exponents are the Fourier coefficients of a modular form (or Jacobi form) for $\mathrm{SL}_2$, and we refer to it as a reflective Borcherds product.

To define modular forms that are singular at a prime $p$ we have to work in the neighborhood of a fixed cusp. Suppose $c \in L$ is a primitive vector of norm $0$ and $c' \in L'$ is an element of the dual lattice with $\langle c, c' \rangle = 1.$ Let $L_{c, c'}$ be the orthogonal complement of $c$ and $c'$, i.e.
$$
L_{c, c'} = \{\lambda \in L: \; \langle \lambda, c \rangle = \langle \lambda, c' \rangle = 0\}.
$$ 
Attached to the pair $(c, c')$ we have the tube domain $$
\mathbb{H}_{c, c'} = \{Z = X+iY \in L_{c, c'} \otimes \mathbb{C}: \; \mathcal{Z} := c' + Z - Q(Z)c \in \mathcal{A}\},
$$ 
which is one of the two connected components of the set 
$$
\{Z = X+iY \in L_{c, c'} \otimes \mathbb{C}: \; \langle Y, Y \rangle < 0\}.
$$
On $\mathbb{H}_{c, c'}$, any modular form $F$ can be written as a Fourier series 
$$
F(Z) = \sum_{\substack{\lambda \in L_{c, c'} \otimes \mathbb{Q}}} a_F(\lambda) e^{2\pi i \langle \lambda, Z\rangle},
$$
in which the actual values of $\lambda$ range over a discrete group depending on $\Gamma$ and the character $\chi$. To be more precise: there exists a sublattice $K$ of $L_{c,c'}$ such that $\lambda$ lies in the dual $K'$ of $K$ whenever $a_F(\lambda)\neq 0$. By definition, the level of $K$ is the smallest positive integer $N$ such that $N\latt{v,v}\in 2\ZZ$ for all $v\in K'$. Clearly, 
$$
N\latt{\lambda,\lambda} \in 2\ZZ  
$$
for all $\lambda\in L_{c,c'}\otimes \QQ$ satisfying $a_F(\lambda)\neq 0$. Therefore, there is a smallest positive integer $D_F$ such that $D_F \cdot Q(\lambda)\in \ZZ$ for all $\lambda\in L_{c,c'}\otimes \QQ$ with $a_F(\lambda)\neq 0$. It is clear that $D_F|N$.

A non-constant modular form $F$ is called \emph{singular} (with respect to the pair $(c, c')$) if its Fourier series on $\mathbb{H}_{c, c'}$ is supported on vectors $\lambda$ of norm zero. By analogy, we define \textit{singular modular forms modulo a prime $p$} as follows:

\begin{definition}
Let $F$ be a non-constant modular form and $p$ be a prime not dividing $D_F$. The form $F$ is called \emph{singular modulo $p$} (at the cusp determined by $(c, c')$) if its Fourier coefficients are all integers and if $$a_F(\lambda) \equiv 0 \quad (\text{mod}\; p)$$ for all vectors $\lambda$ for which $Q(\lambda)$ is nonzero modulo $p$.
\end{definition}

\begin{remark}
Using the Fourier--Jacobi expansion, it is not difficult to show that a modular form is singular if and only if its weight is $k = n/2 - 1$. In this case, it is singular at every cusp. 

\vspace{2mm}

The notion of mod $p$ singular modular forms also appears to be independent of the choice of cusps, and the weight appears to satisfy the similar constraint $$k \equiv (n/2 - 1) \quad (\text{mod} \; p).$$ Unfortunately we do not have a proof of this. The converse is false: most modular forms of weight $k \equiv (n/2-1)$ mod $p$ fail to be singular modulo $p$.
\end{remark}

Singularity with respect to $(c, c')$ is closely related to the \emph{holomorphic Laplace operator}. If $e_1,...,e_n$ is any basis of $L_{c, c'}$ with Gram matrix $S$, and $z_1,...,z_n$ are the associated coordinates on $L_{c, c'} \otimes \mathbb{C}$, then define 
$$
\mathbf{\Delta} = \mathbf{\Delta}_{c, c'} := \frac{1}{8\pi^2} \sum_{i, j=1}^n s^{ij} \frac{\partial^2}{\partial z_i \partial z_j}
$$ 
where $s^{ij}$ are the entries of $S^{-1}$. Note that $\mathbf{\Delta}$ is independent of the basis $e_i$. 

Applying
$$
\mathbf{\Delta} \Big( e^{2\pi i \langle \lambda, Z \rangle} \Big) = -Q(\lambda) e^{2\pi i \langle \lambda, Z \rangle},
$$ 
to the Fourier series termwise shows that the form $F$ is annihilated by $\mathbf{\Delta}$ if and only if it is singular at $(c, c').$ 
Similarly, if $F$ has integral coefficients then $F$ is singular modulo $p$ if and only if 
$$
\mathbf{\Delta} (F) \equiv 0 \quad (\text{mod}\, p);
$$
here we recall that $D_F\cdot \mathbf{\Delta} (F)$ also has integral Fourier coefficients at the cusp $(c,c')$ and that $p$ does not divide $D_F$ by definition. 

The setting of \cite{KKN}, i.e. Siegel modular forms of degree two, corresponds to the case of the lattice $L = 2U \oplus A_1$ i.e.
$$
\mathbb{Z}^5 \; \text{with Gram matrix} \; \begin{psmallmatrix} 0 & 0 & 0 & 0 & 1 \\ 0 & 0 & 0 & 1 & 0 \\ 0 & 0 & 2 & 0 & 0 \\ 0 & 1 & 0 & 0 & 0 \\ 1 & 0 & 0 & 0 & 0 \end{psmallmatrix}.
$$
If we work with $c = (1, 0, 0, 0, 0)$ and $c' = (0, 0, 0, 0, 1)$ then vectors $(0, z_1, z_2, z_3, 0)$ of $\mathbb{H}_{c, c'}$ correspond exactly to matrices 
$\begin{psmallmatrix} 
z_1 & z_2 \\ z_2 & -z_3 
\end{psmallmatrix}$
in the Siegel upper half-space in a way that is compatible with the actions of $\mathrm{O}(3, 2)$ and $\mathrm{Sp}_4(\mathbb{R})$, and the Laplace operator at $(c, c')$ becomes (up to a scalar multiple) the \emph{theta-operator} $\frac{\partial^2}{\partial z_1 \partial z_3} - \frac{\partial^2}{\partial z_2^2}$. See also Section \ref{subsec:Siegel} below.

\section{The construction of singular automorphic products modulo primes}\label{sec:construction}

Let $L$ be an even lattice of signature $(n, 2)$ with $n \ge 3$ that contains a primitive vector $c$ of norm zero and a vector $c' \in L'$ with $\langle c, c'\rangle = 1$. The Laplace operator attached to the pair $(c, c')$ is simply denoted $\mathbf{\Delta}$. Let $\Gamma \leq \mathrm{O}^+(L)$ be a modular group.
Note that $\Gamma$ satisfies Koecher's principle: the Baily--Borel compactification of $\mathcal{D} / \Gamma$ contains no cusps in codimension one.

\begin{lemma} 
For modular forms $F$ of weight $k$ and $G$ of weight $\ell$ for $\Gamma$, the bracket \begin{align*} [F, G] :=& \Big( \frac{n}{2} - 1 - k \Big) \Big( \frac{n}{2} - 1 - \ell \Big) \mathbf{\Delta}(F G) \\ -& \Big( \frac{n}{2} - 1 - \ell \Big) \Big(\frac{n}{2} - 1 - k - \ell \Big) \mathbf{\Delta}(F) G \\ -& \Big( \frac{n}{2} - 1 - k \Big) \Big( \frac{n}{2} - 1 - k - \ell \Big) F \mathbf{\Delta}(G)\end{align*} is a cusp form of weight $k + \ell + 2$.
More generally, if $F$ has character $\chi_F$ and $G$ has character $\chi_G$, then the bracket $[F, G]$ has character $\chi_F \chi_G$.
\end{lemma}
\begin{proof} Up to a scalar multiple, this is the first \emph{Rankin--Cohen bracket} of $F$ and $G$ as defined by Choie and Kim \cite{CK}. The assumption of \cite{CK} that the lattice $L$ splits two hyperbolic planes is unnecessary. This lemma can also be proved directly by analyzing how $\mathbf{\Delta}(F)$ transforms under the modular group. In particular, it follows from \cite[Lemma 2.4]{Wil21} that 
\begin{equation}\label{eq:transformation}
[F, G]|\gamma = [F|\gamma, G|\gamma]    
\end{equation}
for any $\gamma \in \Orth^+(L_\RR)$. 
\end{proof}

Since $F$ is singular modulo $p$ if and only if all Fourier coefficients of $\mathbf{\Delta}(F)$ vanish modulo $p$, we obtain the corollary:

\begin{corollary}\label{cor}
Let $p$ be a prime that divides the numerator of $\frac{n}{2} - 1 - k$. Suppose $G$ is a modular form of weight $\ell$ that is not identically zero modulo $p$. Suppose $p$ does not divide $\ell$ and that $p$ does not divide the numerator of $\frac{n}{2}-1-\ell$. The following are equivalent: 
\begin{enumerate}
    \item $F$ is singular modulo $p$;
    \item The cusp form $[F, G]$ vanishes identically modulo $p$.
\end{enumerate}
\end{corollary}

Now suppose that $F$ is a reflective modular form for $\Gamma \leq \Orth^+(L)$ with only simple zeros, and that $G$ is a modular form for $\Gamma$ that is non-vanishing on every zero $r^{\perp}$ of $F$. Since the associated reflection $\sigma_r$ is an involution and is contained in $\Gamma$, it follows that 
$$
F(\sigma_r \mathcal{Z}) = -F(\mathcal{Z}) \quad \text{and} \quad G(\sigma_r \mathcal{Z}) = G(\mathcal{Z}).
$$ 

Fix a $\mathbb{R}$-basis $e_1,...,e_n$ of the lattice $L_{c,c'}$ for which 
$$
\latt{e_1,e_1}=-1, \quad \latt{e_2, e_2} = \cdots = \latt{e_{n}, e_{n}} = 1.
$$ 
We normalize the vector $r$ such that $\latt{r,r}=1$. 
There exists $\gamma\in \Orth^+(L_\RR)$ such that $\gamma(e_n)=v$, and this $\gamma$ maps the hyperplane $v^{\perp}$ biholomorphically onto $e_n^{\perp}$. 

View $F$ and $G$ as holomorphic functions on $\HH_{c,c'}$. Since $\gamma\sigma_{e_n}\gamma^{-1}=\sigma_{\gamma(e_n)}=\sigma_r$, we have
$$
(F|\gamma)|\sigma_{e_n}=-F|\gamma \quad \text{and}  \quad (G|\gamma)|\sigma_{e_n}=G|\gamma. 
$$

Write the variable $Z\in \HH_{c,c'}$ in the form $z' + z_n e_n$, where $z'\in e_n^\perp $.
Since $\sigma_{e_n}$ fixes $z'$ and maps $z_n$ into $-z_n$, the Taylor series developments of $F|\gamma$ and $G|\gamma$ about $e_n^{\perp}$ have the form 
$$
F|\gamma(Z) = \sum_{m=0}^{\infty} f_{2m+1}(z') z_n^{2m+1} \quad \text{and} \quad G|\gamma(Z) = \sum_{m=0}^{\infty} g_{2m}(z') z_n^{2m},
$$ 
respectively. Applying $\mathbf{\Delta}$, we find that all of $\mathbf{\Delta}(F|\gamma \cdot G|\gamma)$, $\mathbf{\Delta}(F|\gamma)\cdot G|\gamma$ and $F|\gamma \cdot \mathbf{\Delta}(G|\gamma)$ also vanish on the divisor $e_n^{\perp}=\{z_n=0\}$. Using Equation \eqref{eq:transformation}, we find that $[F,G]=[F|\gamma, G|\gamma]|\gamma^{-1}$ vanishes along $r^\perp$. The quotient $\frac{[F, G]}{F}$ is therefore a holomorphic modular form of weight $\ell+2$. 

\vspace{2mm}

If $G$ \emph{also} happens to be a reflective modular form for $\Gamma$, with only simple zeros that are distinct from those of $F$, then the above argument shows that $[F, G]$ is divisible by both $F$ and $G$ and therefore the quotient $\frac{[F, G]}{FG}$ is a holomorphic modular form of weight two without character. 

\vspace{2mm}

Many groups $\Gamma$ do not admit holomorphic modular forms of weight two. (For example, this is always true if $n > 6$, and it is usually true for $\Gamma = \mathrm{O}^+(L)$ if the discriminant of $L$ is reasonably small.) In these cases, we obtain $[F, G] = 0$ and therefore an integral relation among $\mathbf{\Delta}(FG)$, $\mathbf{\Delta}(F)G$ and $F \mathbf{\Delta}(G)$. This is summarized below:

\begin{proposition}\label{prop:main}
Let $L$ be an even lattice of signature $(n,2)$ with $n\geq 3$. 
Suppose $F$ and $G$ are reflective modular forms for $\Gamma\leq \Orth^+(L)$ of weights $k$ and $\ell$ with simple and disjoint zeros, and that $\Gamma$ admits no modular forms of weight two with trivial character. Then we have the identity 
\begin{align*}
\Big( \frac{n}{2} - 1 - k \Big) \Big( \frac{n}{2} - 1 - \ell \Big) \mathbf{\Delta}(F G) 
=& \Big( \frac{n}{2} - 1 - \ell \Big) \Big(\frac{n}{2} - 1 - k - \ell \Big) \mathbf{\Delta}(F) G \\
+& \Big( \frac{n}{2} - 1 - k \Big) \Big( \frac{n}{2} - 1 - k - \ell \Big) F \mathbf{\Delta}(G).    
\end{align*}
In particular,
\begin{enumerate}
\item $F$ is singular modulo every prime dividing $\frac{n}{2} - 1 - k$ but neither $\ell$ nor $\frac{n}{2}-1-\ell$;
\item $G$ is singular modulo every prime dividing $\frac{n}{2} - 1 - \ell$ but neither $k$ nor  $\frac{n}{2}-1-k$;
\item $FG$ is singular modulo every prime dividing $\frac{n}{2} - 1 - (k+\ell)$ but neither $k$ nor $\ell$.
\end{enumerate}
\end{proposition}

More generally, under these assumptions, $F$ is singular modulo any prime $p$ that divides $\frac{n}{2} - 1 - k$ to a greater power than any of $\frac{n}{2}-1-\ell$ and $\frac{n}{2}-1-k-\ell$, and similarly for $G$ and $FG$. 

\begin{remark}
The bracket $[-,-]$ can be generalized to any number of modular forms. Let  $F_1,...,F_N$ be modular forms for $\Gamma\leq \Orth^+(L)$ of weights $k_1,...,k_N$. Then
\begin{align*} 
[F_1,...,F_N] &:= \prod_{i=1}^N\Big(\frac{n}{2}-1-k_i \Big) \times \mathbf{\Delta}\left(\prod_{i=1}^N F_i \right) \\ & \quad - \Big( \frac{n}{2} - 1 - \sum_{i=1}^{N} k_i \Big) \times \sum_{j=1}^N F_j \cdot \mathbf{\Delta}\left( \prod_{i \ne j} \Big(\frac{n}{2} - 1 - k_i \Big) F_i \right) \end{align*} 
defines a cusp form of weight $2 + \sum_{i=1}^N k_i$ for $\Gamma$. This is also a special case of the Rankin--Cohen brackets defined in \cite{CK}. The identity in Proposition \ref{prop:main} generalizes to an identity involving any number of reflective products; however, this does not appear to give any information not already obtained from considering the products in pairs.
\end{remark}

It was proved in \cite{WW23b} that every holomorphic Borcherds product of singular weight on $L$ can be viewed as a reflective modular form, possibly after passing to a distinct lattice in $L\otimes \QQ$. It is amusing that the notion of reflective modular forms plays a similar role for congruences.

\section{Examples}
In this section we use Proposition \ref{prop:main} to produce a number of examples of reflective Borcherds products on orthogonal groups of root lattices or related lattices that are singular modulo certain primes. The non-existence of modular forms of weight two in the nontrivial case of $n\leq 6$ can be derived from \cite{WW20a, WW21a, WW23a}, where the entire graded rings of modular forms were determined. 

We denote by $U$ the hyperbolic plane, i.e. the lattice $\ZZ^2$ with Gram matrix $\begin{psmallmatrix} 0 & 1  \\ 1 & 0 \end{psmallmatrix}$. Let $A_n$, $D_n$, $E_6$, $E_7$ and $E_8$ be the usual root lattices. For a lattice $L$ and $d \in \mathbb{N}$, we write $L(d)$ to mean $L$ with its quadratic form multiplied by the factor $d$.

\subsection{Siegel modular forms of degree two}\label{subsec:Siegel}
When $L$ is the lattice $2U \oplus A_1$ with $n = 3$, modular forms for $\mathrm{O}^+(L)$ are the same as Siegel modular forms of degree two and even weight for the level one modular group $\mathrm{Sp}_4(\mathbb{Z}).$ Through this identification, rational quadratic divisors become the classical Humbert surfaces defined by singular relations. There are two equivalence classes of reflective divisors: 
\begin{itemize}
\item[(i)] The Humbert surface of invariant one, which is represented by the set of diagonal matrices $\begin{psmallmatrix} \tau & 0 \\ 0 & w \end{psmallmatrix}$ in $\mathbb{H}_2$; 
\item[(ii)] The Humbert surface of invariant four, which is represented by the set of matrices $\begin{psmallmatrix} \tau & z \\ z & \tau \end{psmallmatrix}$ with equal diagonal entries. 
\end{itemize}
Both reflective Humbert surfaces occur as the zero locus of a Borcherds product for $\Orth^+(L)$:
\begin{itemize}
\item[(a)] The form $\Psi_5$ of weight $k = 5$, a square root of the Igusa cusp form of weight $10$, vanishes with simple zeros on the Humbert surface of invariant one;
\item[(b)] The quotient $\Phi_{30} = \Phi_{35} / \Psi_5$ of weight $\ell = 30$, where $\Phi_{35}$ is the cusp form of weight $35$, vanishes with simple zeros on the Humbert surface of invariant four.
\end{itemize}

We calculate 
\begin{align*} \frac{n}{2} - 1 - k &= -\frac{9}{2}; \\ 
\frac{n}{2}-1-\ell &= -\frac{59}{2}; \\ 
\frac{n}{2} - 1 - k - \ell &= -\frac{69}{2} = -\frac{3 \cdot 23}{2}.
\end{align*} 
Proposition \ref{prop:main} and the non-existence of Siegel modular forms of weight two yields: 
\begin{enumerate}
\item $\Psi_5$ is singular modulo $p=3$;
\item $\Phi_{30}$ is singular modulo $p=59$;
\item $\Phi_{35}=\Psi_5\Phi_{30}$ is singular modulo $p=23$.
\end{enumerate}

\subsection{Siegel paramodular forms of degree two and level 2 and 3} 
Section \ref{subsec:Siegel} gives the simplest example of a number of realizations of arithmetic subgroups of $\mathrm{Sp}_4(\mathbb{Q})$ as orthogonal groups of lattices. When $L = 2U \oplus A_1(t)$, modular forms for $\mathrm{O}^+(L)$ are the same as Siegel paramodular forms of degree two and level $t$ that are invariant under certain additional involutions. We will work out the congruences implied by Proposition \ref{prop:main} when $t=2$ or $t=3$. 

\subsubsection{Level $2$}
When $t = 2$, there are three equivalence classes of reflective divisors associated to three reflective Borcherds products for $\Orth^+(L)$: 
\begin{itemize}
\item[(i)] $\Psi_2$, a weight two cusp form with a character $\chi$ of order four that vanishes precisely on hyperplanes $r^{\perp}$ with $r\in L'$ and $Q(r) = 1/8$;
\item[(ii)] $\Psi_9$, a weight nine cusp form with character $\chi^3$ that vanishes precisely on primitive hyperplanes $r^{\perp}$ with $r\in L'$ and $Q(r) = 1/2$;
\item[(iii)] $\Phi_{12}$, a weight twelve non-cusp form that vanishes precisely on hyperplanes $r^{\perp}$ with $r\in L$ and $Q(r) = 1$.
\end{itemize}
There are no (nonzero) modular forms of weight 2 with trivial character for $\Orth^+(2U\oplus A_1(2))$. By applying Proposition \ref{prop:main} to all pairs that can be formed from $\{\Psi_2, \Psi_9, \Phi_{12}\}$, we find:
\begin{enumerate}
\item $\Psi_9$ is singular modulo $p=17$; 
\item $\Phi_{12}$ is singular modulo $p=23$; 
\item $\Psi_2 \Psi_9$ is singular modulo $p=7$; 
\item $\Psi_2\Phi_{12}$ is singular modulo $p=3$; 
\item $\Psi_9 \Phi_{12}$ is singular modulo $p=41$; 
\item $\Psi_2 \Psi_9 \Phi_{12}$ is singular modulo $p=5$. 
\end{enumerate}

Proposition \ref{prop:main} does not imply any congruence for $\Psi_2$ itself, and judging by its first Fourier coefficients, $\Psi_2$ does not appear to be singular modulo any prime.

\subsubsection{Level $3$} In level $t=3$, there are also three equivalence classes of reflective divisors associated to three reflective Borcherds products for $\Orth^+(L)$: 
\begin{itemize}
\item[(i)] $\Psi_1$, a weight one cusp form with a character $\chi$ of order six that vanishes precisely on hyperplanes $r^{\perp}$ with $r\in L'$ and $Q(r) = 1/12$;
\item[(ii)] $\Psi_6$, a weight six cusp form with character $\chi^3$ that vanishes precisely on primitive hyperplanes $r^{\perp}$ with $r\in L'$ and $Q(r) = 1/3$;
\item[(iii)] $\Phi_{12}$, a weight twelve non-cusp form that vanishes precisely on hyperplanes $r^{\perp}$ with $r\in L$ and $Q(r) = 1$.
\end{itemize}
By applying Proposition \ref{prop:main} to the three pairs of reflective Borcherds products we obtain:
\begin{enumerate}
\item $\Psi_6$ is singular modulo $p=11$; 
\item $\Phi_{12}$ is singular modulo $p=23$; 
\item $\Psi_1 \Psi_6$ is singular modulo $p = 13$; 
\item $\Psi_1 \Phi_{12}$ is singular modulo $p = 5$; 
\item $\Psi_6 \Phi_{12}$ is singular modulo both $p=5$ and $p=7$; 
\item $\Psi_1 \Psi_6 \Phi_{12}$ is singular modulo $p=37$.
\end{enumerate}

Similarly to the level 2 case, we are unable to obtain any congruence for $\Psi_1$ itself; and indeed $\Psi_1$ does not appear to be singular modulo any prime.

\subsection{Hermitian modular forms of degree two over the Eisenstein integers}
Modular forms for the orthogonal group of $L = 2U \oplus A_2$ are essentially the same as Hermitian modular forms of degree two for the full modular group over the Eisenstein integers. There are two classes of reflective divisors associated to two Borcherds products for $\Orth^+(L)$: 
\begin{itemize}
\item[(i)] $\Psi_9$, a cusp form of weight $9$ with a quadratic character that vanishes precisely on hyperplanes $r^{\perp}$ with $r\in L'$ and $Q(r) = 1/3$;
\item[(ii)] $\Phi_{45}$, a cusp form of weight $45$ with a quadratic character that vanishes precisely on hyperplanes $r^{\perp}$ with $r\in L$ and $Q(r) = 1$.
\end{itemize}
We derive from Proposition \ref{prop:main} the following
\begin{enumerate}
\item $\Psi_9$ is singular modulo $p = 2$ (see the paragraph after Proposition \ref{prop:main});
\item $\Phi_{45}$ is singular modulo $p=11$;
\item $\Psi_9 \Phi_{45}$ is singular modulo $p=53$.
\end{enumerate}

The congruences (1) and (2) were proved in \cite{NT} using an argument based on a Sturm bound for Hermitian modular forms.

\subsection{Modular forms on \texorpdfstring{$2U(2)\oplus A_2$}{}}
Let $L=2U(2)\oplus A_2$. The orthogonal group $\Orth^+(L)$ can be realized as a level two subgroup of the Hermitian modular group of degree two over the Eisenstein integers (cf. \cite{HK} for a precise statement). There are three classes of reflective hyperplanes, each occurring as the divisor of a Borcherds product for $\Orth^+(L)$ with simple zeros: 
\begin{itemize}
\item[(i)] $\Psi_{3}$, a modular form of weight $3$ that vanishes precisely on hyperplanes $r^\perp$ with $r\in L'$, $Q(r)=1/3$ and $3r\in L$;
\item[(ii)] $\Psi_{12}$, a modular form of weight $12$ that vanishes precisely on hyperplanes $r^\perp$ with $r\in L'$, $Q(r)=1/2$ and $2r\in L$;
\item[(iii)] $\Phi_{15}$, a modular form of weight $15$ that vanishes precisely on hyperplanes $r^\perp$ with $r\in L$, $Q(r)=1$ and $r/2\not\in L'$.
\end{itemize}
By Proposition \ref{prop:main}, we find:
\begin{enumerate}
\item $\Psi_{12}$ is singular modulo $p=11$;
\item $\Phi_{15}$ is singular modulo $p=7$;
\item $\Psi_3\Psi_{12}$ is singular modulo $p=7$;
\item $\Psi_3\Phi_{15}$ is singular modulo $p=17$;
\item $\Psi_{12}\Phi_{15}$ is singular modulo $p=13$;
\item $\Psi_3\Psi_{12}\Phi_{15}$ is singular modulo $p=29$.
\end{enumerate}

\subsection{Hermitian modular forms of degree two over the Gaussian integers}
Modular forms for the orthogonal group of $L=2U\oplus 2A_1$ are essentially the same as Hermitian modular forms of degree two for the full modular group over the Gaussian integers. There are three classes of reflective divisors associated to three products for $\Orth^+(L)$: 
\begin{itemize}
\item[(i)] $\Psi_4$, a cusp form of weight four with a quadratic character that vanishes precisely on hyperplanes $r^{\perp}$ with $r\in L'$ and $Q(r) = 1/4$;
\item[(ii)] $\Psi_{10}$, a cusp form of weight ten with a quadratic character that vanishes precisely on hyperplanes $r^{\perp}$ with $r\in L'$ and $Q(r) = 1/2$;
\item[(iii)] $\Phi_{30}$, a skew-symmetric cusp form of weight 30 with a quadratic character that vanishes precisely on hyperplanes $r^{\perp}$ with $r\in L$, $Q(r) = 1$ and $r/2 \not\in L'$.
\end{itemize}
Note that $\Psi_4 \Psi_{10} \Phi_{30}$ has the determinant character on $\Orth^+(L)$. 
By applying Proposition \ref{prop:main}, 
\begin{enumerate}
\item $\Psi_{10}$ is singular modulo $p=3$; 
\item $\Phi_{30}$ is singular modulo $p=29$; 
\item $\Psi_4 \Psi_{10}$ is singular modulo $p=13$; 
\item $\Psi_4 \Phi_{30}$ is singular modulo $p=11$; 
\item $\Psi_{10} \Phi_{30}$ is singular modulo $p=13$; 
\item $\Psi_4 \Psi_{10} \Phi_{30}$ is singular modulo $p=43$. 
\end{enumerate}

\subsection{Modular forms on \texorpdfstring{$2U\oplus A_3$}{}}
Let $L=2U\oplus A_3$. There are two classes of reflective hyperplanes associated to two reflective Borcherds products for $\Orth^+(L)$ with simple zeros: 
\begin{itemize}
\item[(i)] $\Psi_9$, a cusp form of weight $9$ with a quadratic character that vanishes precisely on hyperplanes $r^\perp$ with $r\in L'$ and $Q(r)=1/2$;  
\item[(ii)] $\Phi_{54}$, a cusp form of weight $54$ with a quadratic  character that vanishes precisely on hyperplanes $r^\perp$ with $r\in L$ and $Q(r)=1$. 
\end{itemize}
Proposition \ref{prop:main} then yields:
\begin{enumerate}
\item $\Phi_{54}$ is singular modulo $p=7$;
\item $\Psi_9\Phi_{54}$ is singular modulo $p=41$.
\end{enumerate}

\subsection{Modular forms on \texorpdfstring{$2U\oplus D_4$}{}}
Let $L = 2U \oplus D_4$. Modular forms for the orthogonal group of $L$ are essentially the same as modular forms for the quaternionic modular group attached to the order of Hurwitz integers as defined in \cite{K}. There are two orbits of reflective hyperplanes which belong to two reflective Borcherds products on $\Orth^+(L)$ with simple zeros: 
\begin{itemize}
\item[(i)] $\Psi_{24}$, a cusp form of weight $24$ with a quadratic character that vanishes precisely on hyperplanes $r^\perp$ with $r\in L'$ and $Q(r)=1/2$;
\item[(ii)] $\Phi_{72}$, a cusp form of weight $72$ with a quadratic character that vanishes precisely on hyperplanes $r^\perp$ with $r\in L$ and $Q(r)=1$.
\end{itemize}
By Proposition \ref{prop:main},
\begin{enumerate}
\item $\Phi_{72}$ is singular modulo both $p=5$ and $p=7$;
\item $\Psi_{24}\Phi_{72}$ is singular modulo $p=47$.
\end{enumerate}

\vspace{2mm}

Fix the model $$D_4=\{ (x_1,x_2,x_3,x_4) \in \ZZ^4 : x_1+x_2+x_3+x_4 \in 2\ZZ \}.$$ There exists a Borcherds product of weight $8$ that vanishes precisely on hyperplanes $r^\perp$ with $r \in (1,0,0,0) + D_4$ and $Q(r)=1/2$, and we denote it by $\Psi_8$. Let $\Gamma$ be the subgroup of $\Orth^+(L)$ generated by the reflections through hyperplanes in the divisor of $\Psi_8\Phi_{72}$. Then $\Psi_8$ and $\Phi_{72}$ are modular under $\Gamma$ and are therefore reflective modular forms for $\Gamma$. By applying Proposition \ref{prop:main}, we find:
\begin{itemize}
\item $\Psi_{8}\Phi_{72}$ is singular modulo $p=13$.
\end{itemize}

There are also products $\Psi_8^{(1)}$ and $\Psi_8^{(2)}$ that vanish exactly on the hyperplanes $r^\perp$ with $Q(r) = 1$ and $r\in (1/2,1/2,1/2,1/2)+L$ or $r\in (-1/2,1/2,1/2,1/2)+L$, respectively, and $\Psi_{24}$ factors as $$\Psi_{24}=\Psi_8\Psi_8^{(1)}\Psi_8^{(2)}.$$ However, Proposition \ref{prop:main} does not apply to these factors because neither $\Psi_8^{(1)}$ nor $\Psi_8^{(2)}$ is modular under the reflections through their zeros; that is, neither $\Psi_8^{(1)}$ nor $\Psi_8^{(2)}$ is reflective in the sense of Definition \ref{def:reflective}.

\subsection{Modular forms on \texorpdfstring{$2U\oplus 2A_2$}{}}
Let $L$ be the signature $(8, 2)$ lattice $2U\oplus 2A_2$. There are two classes of reflective divisors associated to two reflective Borcherds products for $\Orth^+(L)$: 
\begin{itemize}
\item[(i)] $\Psi_{6}$, a cusp form of weight $6$ that vanishes precisely on  hyperplanes $r^\perp$ with $r\in L'$ and $Q(r)=1/3$;
\item[(ii)] $\Phi_{42}$, a cusp form of weight $42$ that vanishes precisely on hyperplanes $r^\perp$ with $r\in L$ and $Q(r)=1$.
\end{itemize}
By Proposition \ref{prop:main},
\begin{enumerate}
\item $\Phi_{42}$ is singular modulo both $p=2$ and $p=5$;
\item $\Psi_6\Phi_{42}$ is singular modulo $p=23$.
\end{enumerate}

\subsection{Modular forms on \texorpdfstring{$2U\oplus D_5$}{}}
Let $L=2U\oplus D_5$. There are two classes of reflective divisors associated to two reflective Borcherds products for $\Orth^+(L)$: 
\begin{itemize}
\item[(i)] $\Psi_{7}$, a cusp form of weight $7$ with a quadratic character that vanishes precisely on hyperplanes $r^\perp$ with $r\in L'$ and $Q(r)=1/2$;
\item[(ii)] $\Phi_{88}$, a cusp form of weight $88$ with a quadratic  character that vanishes precisely on hyperplanes $r^\perp$ with $r\in L$ and $Q(r)=1$.
\end{itemize}
By Proposition \ref{prop:main},
\begin{enumerate}
\item $\Phi_{88}$ is singular modulo $p=19$;
\item $\Psi_{7}\Phi_{88}$ is singular modulo both $p=5$ and $p=37$.
\end{enumerate}

\subsection{Modular forms on \texorpdfstring{$2U\oplus D_6$}{}}
Let $L=2U\oplus D_6$. There are two classes of reflective hyperplanes, each occuring as the divisor of a reflective Borcherds product for $\Orth^+(L)$ with simple zeros: 
\begin{itemize}
\item[(i)] $\Psi_{6}$, a cusp form of weight $6$ with a quadratic character that vanishes precisely on hyperplanes $r^\perp$ with $r\in L'$ and $Q(r)=1/2$;
\item[(ii)] $\Phi_{102}$, a cusp form of weight $102$ with a quadratic  character that vanishes precisely on hyperplanes $r^\perp$ with $r\in L$ and $Q(r)=1$.
\end{itemize}
By Proposition \ref{prop:main},
\begin{enumerate}
\item $\Phi_{102}$ is singular modulo both $p=3$ and $p=11$;
\item $\Psi_{6}\Phi_{102}$ is singular modulo both $p=5$ and $p=7$.
\end{enumerate}

\subsection{Modular forms on \texorpdfstring{$2U\oplus E_6'(3)$}{}}
Let $L=2U\oplus E_6'(3)$, or equivalently $2U(3) \oplus E_6$. There are two classes of reflective hyperplanes and each occurs as the divisor of a reflective Borcherds product for $\Orth^+(L)$ with simple zeros: 
\begin{itemize}
\item[(i)] $\Psi_{12}$, a non-cusp form of weight $12$ that vanishes precisely on hyperplanes $r^\perp$ with $r\in L'$ and $Q(r)=1/3$;
\item[(ii)] $\Phi_{12}$, a non-cusp form of weight $12$ that vanishes precisely on hyperplanes $r^\perp$ with $r\in L$ and $Q(r)=1$.
\end{itemize}
By Proposition \ref{prop:main},
\begin{itemize}
\item $\Psi_{12}\Phi_{12}$ is singular modulo $p=7$.
\end{itemize}

\subsection{Modular forms on \texorpdfstring{$2U\oplus 2A_3$}{}}
Let $L=2U\oplus 2A_3$. There are two classes of reflective hyperplanes under $\Orth^+(L)$, each occurring as the divisor of a reflective Borcherds product for $\Orth^+(L)$: 
\begin{itemize}
\item[(i)] $\Psi_{6}$, a cusp form of weight $6$ that vanishes precisely on hyperplanes $r^\perp$ with $r\in L'$ and $Q(r)=1/2$;
\item[(ii)] $\Phi_{48}$, a cusp form of weight $48$ that vanishes precisely on hyperplanes $r^\perp$ with $r\in L$, $Q(r)=1$ and $r/2\not\in L'$.
\end{itemize}
By applying Proposition \ref{prop:main} we obtain
\begin{enumerate}
\item $\Phi_{48}$ is singular modulo both $p=3$ and $p=5$;
\item $\Psi_6\Phi_{48}$ is singular modulo $p=17$.
\end{enumerate}

\subsection{Modular forms on \texorpdfstring{$2U\oplus D_7$}{}}
Let $L=2U\oplus D_7$. There are two classes of reflective hyperplanes, each occuring as the divisor of a reflective Borcherds product for $\Orth^+(L)$: 
\begin{itemize}
\item[(i)] $\Psi_{5}$, a cusp form of weight $5$ with a quadratic character that vanishes precisely on hyperplanes $r^\perp$ with $r\in L'$ and $Q(r)=1/2$;
\item[(ii)] $\Phi_{114}$, a cusp form of weight $114$ with a quadratic  character that vanishes precisely on hyperplanes $r^\perp$ with $r\in L$ and $Q(r)=1$.
\end{itemize}
By Proposition \ref{prop:main},
\begin{enumerate}
\item $\Phi_{114}$ is singular modulo both $p=13$ and $p=17$;
\item $\Psi_{5}\Phi_{114}$ is singular modulo both $p=7$ and $p=11$.
\end{enumerate}

\subsection{Modular forms on \texorpdfstring{$2U\oplus E_8(2)$}{}}
Let $L=2U\oplus E_8(2)$. There are two classes of reflective hyperplanes and each occurs as the divisor of a reflective Borcherds product for $\Orth^+(L)$: 
\begin{itemize}
\item[(i)] $\Psi_{60}$, a cusp form of weight $60$ that vanishes precisely on hyperplanes $r^\perp$ with $r\in L'$ and $Q(r)=1/2$;
\item[(ii)] $\Phi_{12}$, a non-cusp form of weight $12$ that vanishes precisely on hyperplanes $r^\perp$ with $r\in L$ and $Q(r)=1$.
\end{itemize}
By Proposition \ref{prop:main},
\begin{enumerate}
\item $\Psi_{60}$ is singular modulo $p=7$;
\item $\Psi_{60}\Phi_{12}$ is singular modulo $p=17$.
\end{enumerate}

\subsection{Modular forms on \texorpdfstring{$2U\oplus D_8'(2)$}{}}
Let $L=2U\oplus D_8'(2)$, or equivalently $2U(2) \oplus D_8$. There are  two  reflective Borcherds product for $\Orth^+(L)$ with simple zeros: 
\begin{itemize}
\item[(i)] $\Psi_{28}$, a cusp form of weight $28$ that vanishes precisely on hyperplanes $r^\perp$ with $r\in L'$ and $Q(r)=1/2$;
\item[(ii)] $\Phi_{28}$, a cusp form of weight $28$ that vanishes precisely on hyperplanes $r^\perp$ with $r\in L$ and $Q(r)=1$.
\end{itemize}
By Proposition \ref{prop:main},
\begin{itemize}
\item $\Psi_{28}\Phi_{28}$ is singular modulo $p=13$.
\end{itemize}

\section{Further examples}

Corollary \ref{cor} can be applied to prove that a reflective modular form with simple zeros is singular modulo various primes even when that form is unique.
This is the case for the root lattices $L = 2U \oplus E_d$ for $d=6,7,8$, where there is a unique reflective form $\Phi_k$ of weight $k=120, 165, 252$, respectively. These forms are Borcherds products that vanish precisely on hyperplanes $r^{\perp}$ with $r \in L$ and $Q(r) = 1$.

Let $h$ be the Coxeter number of $E_d$: so $h=12$ for $E_6$, $h=18$ for $E_7$, and $h=30$ for $E_8$. Let $\rho$ be the Weyl vector of $E_d$. The Fourier expansion of  $\Phi_k$ about the $1$-dimensional cusp determined by $2U$ takes the form
$$
\Phi_k(Z)=\sum_{m=h}^\infty \sum_{n=h}^\infty \sum_{v\in E_d'} f(n,v,m) q^n \zeta^v \xi^m, 
$$
where $Z=(\tau, \mathfrak{z}, \omega) \in \HH \times (E_d\otimes\CC) \times \HH$, $q=e^{2\pi i\tau}$, $\zeta^v = e^{2\pi i\latt{\mathfrak{z},v}}$ and $\xi=e^{2\pi i\omega}$. Since $\Phi_k$ is skew-invariant under the involution $(\tau, \mathfrak{z}, \omega)\mapsto (\omega, \mathfrak{z}, \tau)$, its Fourier coefficients satisfy $$f(n,v,m)=-f(m,v,n),$$ and in particular $f(n,v,n)=0$ for all $n\in \ZZ$ and $v\in E_d'$. Therefore, $q^{h+1}\zeta^\rho\xi^h$ is the leading term with coefficient $1$ in the Fourier expansion of $\Phi_k$. This means that the ``Weyl vector" in the sense of Borcherds products of $\Phi_k$ is $(h+1,\rho,h)$.

By \cite{WW20a}, in all three cases there exists $\ell \in \mathbb{N}$, not equal to the singular weight, such that:
\begin{itemize}
\item[(i)] There exists a modular form $G_\ell$ with trivial character for $\Orth^+(L)$ whose Fourier-Jacobi expansion begins $G_\ell=1+O(\xi)$;
\item[(ii)] The space $M_{\ell+2}(\Orth^+(L))$ is one-dimensional, generated by a modular form $G_{\ell+2}$ whose Fourier-Jacobi expansion begins $G_{\ell+2}=1+O(\xi)$.
\end{itemize}
For $E_6$ and $E_7$ we can take $\ell = 4$ and for $E_8$ we can take $\ell=8$. Moreover, both $G_\ell$ and $G_{\ell+2}$ can be chosen to be Eisenstein series, all of whose Fourier coefficients are integers. 

The form $G_\ell$ certainly does not vanish everywhere on the zero locus of $\Phi_k$, as $G_\ell / \Phi_k$ would otherwise be a holomorphic modular form of weight $\ell-k<0$, which is impossible. By the discussion in Section \ref{sec:construction}, 
$$
[\Phi_k, G_\ell] / \Phi_k \in M_{\ell+2}(\Orth^+(L)). 
$$ 
By construction, there is a constant $c$ such that 
\begin{equation}\label{eq:identity}
[\Phi_k, G_\ell] = c \Phi_k G_{\ell+2}.     
\end{equation}
Comparing Fourier coefficients shows that $c$ is the coefficient of the leading term $q^{h+1}\zeta^\rho\xi^h$ in the Fourier expansion of $[\Phi_k, G_\ell]$. By the definition of the bracket $[-,-]$, we calculate
\begin{align*}
c=&(d/2-k)(d/2-\ell)[h(h+1)-Q(\rho)] - (d/2-k-\ell)(d/2-\ell)[h(h+1)-Q(\rho)] \\
=&\ell(d/2-\ell)[h(h+1)-Q(\rho)]. 
\end{align*}

Equation \eqref{eq:identity} implies the following congruences:
\begin{theorem}
\noindent
\begin{enumerate}
\item $\Phi_{120}$ is singular modulo $p=13$;
\item $\Phi_{165}$ is singular modulo both $p=17$ and $p=19$;
\item $\Phi_{252}$ is singular modulo $p=31$.
\end{enumerate}    
\end{theorem}
\begin{proof}
Recall that $Q(\rho)=\frac{1}{24}h(h+1)d$ and therefore 
$$
h(h+1)-Q(\rho) = \frac{1}{24}h(h+1)(24-d).
$$
Let $p$ be a prime dividing both $d/2-k$ and $h(h+1)-Q(\rho)$ but neither $d/2-\ell$ nor $\ell$. Then \eqref{eq:identity} implies that $\Phi_k$ is singular modulo $p$. We will work out the case of $E_6$ in detail; the other two cases are similar.

In the $E_6$ case, the Rankin--Cohen bracket of $\Phi_{120}$ with the Eisenstein series $G_4$ is
\begin{align*}
[\Phi_{120}, G_4] &=117\times \mathbf{\Delta}(\Phi_{120}G_4) - 121\times 117 \times \mathbf{\Delta}(G_4)\Phi_{120} - 121\times \mathbf{\Delta}(\Phi_{120})G_4\\
&=-36\times 13 \times \Phi_{120}G_6. 
\end{align*}
In particular, $\mathbf{\Delta}(\Phi_{120})G_4 \equiv 0 \; (\mathrm{mod}\; 13)$ and therefore $\mathbf{\Delta}(\Phi_{120}) \equiv 0 \; (\mathrm{mod}\; 13)$, so $\Phi_{120}$ is singular modulo $p=13$. 
\end{proof}

\begin{remark}
Let $\mathfrak{M}_7$ be the unique normalized modular form of weight $7$ for $\Orth^+(2U\oplus E_6)$ which was defined in \cite[Theorem 5.4]{WW20a}. Since there are no nonzero modular form of weight $9$ for $\Orth^+(2U\oplus E_6)$, we conclude 
$$
[\Phi_{120}, \mathfrak{M}_7] = 0. 
$$
This implies that $\Phi_{120}\mathfrak{M}_7$ is singular modulo $p=31$ and also that $\Phi_{120}$ is singular modulo $p=13$. Note that neither $\mathfrak{M}_7$ nor $\Phi_{120}\mathfrak{M}_7$ is a Borcherds product.
\end{remark}

\vspace{2mm}

We conclude with an example of a mod $p$ singular Borcherds product that is not reflective and also has non-simple zeros.

Let $L=2U\oplus D_{11}$ and consider the following two Borcherds products for $\Orth^+(L)$:
\begin{itemize}
\item[(i)] $\Psi_1$, a meromorphic modular form of weight $1$ which vanishes precisely with multiplicity $1$ on hyperplanes $r^\perp$ with $r\in L'$ and $Q(r)=1/2$ and whose only singularities are simple poles along hyperplanes $s^\perp$ with $s\in L'$ and $Q(s)=3/8$; 
\item[(ii)] $\Phi_{142}$, a cusp form of weight $142$ which vanishes precisely with multiplicity $1$ on hyperplanes $\lambda^\perp$ with $\lambda\in L$ and $Q(\lambda)=1$, and with multiplicity $26$ on hyperplanes $s^\perp$ with $s\in L'$ with $Q(s)=3/8$. 
\end{itemize}
The form $\Phi_{142}$ is the Jacobi determinant of the generators of a free algebra of meromorphic modular forms constructed in \cite{WW21a}.
The divisors $r^\perp$ and $\lambda^\perp$ are reflective, i.e. the associated reflections lie in $\Orth^+(L)$. However, the divisors  $s^\perp$ are not reflective. By analyzing its Taylor series along the divisor $s^\perp$, we find that $$\frac{[\Phi_{142}, \Psi_1]}{\Phi_{142}\Psi_1}$$ is a meromorphic modular form of weight $2$ with trivial character for $\Orth^+(L)$ whose only singularities are poles of multiplicity two along the hyperplanes $s^\perp$. By the structure theorem of \cite[Theorem 1.2]{WW21a}, there is a constant $c$ such that
$$
[\Phi_{142}, \Psi_1] = c \cdot \Phi_{142} \Psi_1^3;
$$
that is,
$$
-\frac{273\times 9}{4}\times \mathbf{\Delta}(\Phi_{142}\Psi_1) + \frac{275\times 9}{4}\times \mathbf{\Delta}(\Phi_{142})\Psi_1 - \frac{275\times 273}{4}\times \mathbf{\Delta}(\Psi_1)\Phi_{142} = c \Phi_{142} \Psi_1^3. 
$$
By comparing the residues, or leading terms in the Laurent series of both sides along $s^\perp$, we find that $c=1950$. Therefore:

\begin{theorem}
\noindent
\begin{enumerate}
\item $\Phi_{142}$ is singular modulo $p=13$;
\item $\Psi_{1}\Phi_{142}$ is singular modulo $p=5$.
\end{enumerate}    
\end{theorem}

\bibliographystyle{plainnat}
\bibliofont
\bibliography{refs}

\end{document}